\documentclass[12pt]{amsart}
\usepackage{graphicx}
\usepackage{amssymb}
\usepackage{bm}

\newtheorem{theorem}{Theorem}
\newtheorem{lemma}[theorem]{Lemma}
\newtheorem{proposition}[theorem]{Proposition}

\theoremstyle{definition}
\newtheorem{definition}[theorem]{Definition}
\newtheorem{example}[theorem]{Example}
\newtheorem{notation}[theorem]{Notation}

\theoremstyle{remark}
\newtheorem{remark}[theorem]{Remark}

\newcommand{\C}{\mathbb{C}}

\newcommand{\myotimes}{\mathop{\textstyle{\bigotimes}}}

\newcommand{\N}{\mathbb{N}}

\newcommand{\Tr}{\operatorname{Tr}}

\def\NN{\mathbb{N}}

\newcommand{\ds}{\displaystyle}
\newcommand{\cP}{\mathcal{P}}

\theoremstyle{definition}

\newcommand{\ab}{\allowbreak}

\newcommand{\HH}{\mathcal{H}}
\newcommand{\jj}{{\boldsymbol{j}}}%{\mathbf{j}}

\newcommand{\inp}{{\rm in}}
\newcommand{\outp}{{\rm out}}

\newcommand{\cl}{\mathfrak{l}}
\newcommand{\cF}{\mathfrak{F}}
\newcommand{\cG}{\mathfrak{G}}

\newcommand{\HHs}{\HH^{\blacktriangleright}}
\newcommand{\HHt}{\HH^{\blacktriangleleft}}

\newcommand{\rexp}{\mathfrak{r}}

\begin{document}

\title[sharp Bounds for Graphs of Matrices]
{Sharp Bounds for Sums Associated to Graphs of Matrices}

\author{James A. Mingo}
\address{Department of Mathematics and Statistics\\
Queen's University\\ \ab
Kingston, Ontario K7L 3N6, Canada}

\author{Roland Speicher}
\email{mingo@mast.queensu.ca, speicher@mast.queensu.ca}

\thanks{The authors were supported by Discovery Grants from NSERC}

\date{\today}

\begin{abstract}
We provide a simple algorithm for finding the optimal
upper bound for sums of products of matrix entries of the
form
\[
S_\pi(N) :=
\mathop{\sum_{j_1, \dots, j_{2m} = 1}}_{\ker \jj \geq \pi}^N 
t^{(1)}_{j_1j_2} t^{(2)}_{j_3j_4} \cdots t^{(m)}_{j_{2m-1}j_{2m}}
\]
where some of the summation indices are constrained to be
equal. The upper bound is easily obtained from a graph
$G_\pi$ associated to the constraints, $\pi$, in the sum.
\end{abstract}

\maketitle

\section{Introduction}

We want to consider sums of the form
\begin{equation}\label{eq:S-pi}
S_\pi(N):=\sum_{\substack{j_1,\dots,j_{2m}=1\\ \ker \jj\geq\pi}}^N
t^{(1)}_{j_1j_2}t^{(2)}_{j_3j_4}\cdots t^{(m)}_{j_{2m-1}j_{2m}},
\end{equation}
where $T_k=(t^{(k)}_{ij})_{i,j=1}^N$ ($k=1,\dots,m$) are
given matrices and $\pi$ is a partition of
$\{1,2,\dots,2m\}$ which constrains some of the indices
$j_1,\dots,j_{2m}$ to be the same.

The formal definition of this is given in the following notation.

\begin{notation}
1) A \emph{partition} $\pi=\{V_1,\dots,V_r\}$ of
$\{1,\dots,k\}$ is a decomposition of $\{1,\dots,k\}$ into
disjoint non-empty subsets $V_i$; the $V_i$ are called the
\emph{blocks} of $\pi$. The set of all partitions of
$\{1,\dots,k\}$ is denoted by $\cP(k)$.

2) For $\pi,\sigma\in\cP(k)$, we write $\pi\geq \sigma$ if
each block of $\pi$ is a union of some blocks of $\sigma$.

3) For a multi-index $\jj=(j_1,\dots,j_k)$ we denote by
$\ker \jj\in\cP(k)$ that partition where $p$ and $q$ are in the same block if and only if $j_p = j_q$.

%whose blocks are given by
%positions with the same indices, i.e., for all $p,q=1,\dots,k$

%$$\text{$p$ and $q$ are in the same block of $\ker \jj$} 
%\qquad\Longleftrightarrow\qquad j_p=j_q.$$

\end{notation}

Thus, for a given partition $\pi\in\cP(k)$, the constraint
$\ker\jj\geq \pi$ in \eqref{eq:S-pi} means that two indices
$j_q$ and $j_p$ have to agree, whenever $q$ and $p$ are in
the same block of $\pi$. Note that we do not exclude that
more indices might agree.

The problem which we want to address is the optimal bound
of the sum \eqref{eq:S-pi}.  One expects a
bound of the form
\begin{equation}\label{eq:general}
\vert S_\pi(N)\vert \leq N^{\rexp(\pi)} \prod_{k=1}^m \Vert
T_k\Vert,
\end{equation}
for some exponent $\rexp(\pi)$, where $\|T\|$ denotes the operator norm of the matrix $T$. The question is: what is the
optimal choice of this exponent?

Our interest in sums of the form $S_\pi(N)$ was aroused by
investigations on random matrices where such sums show up
quite canonically, see \cite{MSp}. Indeed, when one considers the asymptotic properties of products of random and deterministic matrices, one has to find efficient bounds for the sums, $S_\pi(N)$, of products of entries of the deterministic matrices in order to determine their contribution to the limiting distribution.
Yin and Krishnaiah
\cite{YK}, working on the product of random matrices,
already faced this problem and obtained the first results
for some special cases; a more systematic approach was given
by Bai \cite{B}. Our investigations are inspired by the
presentation in the book of Bai and Silverstein \cite{BS}.

A first upper bound comes from the trivial observations that
we have in $S_\pi(N)$ one free summation index for each
block of $\pi$ and that $|t^{(k)}_{ij}| \leq \| T_k \|$ for
all $i,j$, and thus one clearly has \eqref{eq:general} with
$\rexp(\pi)=\vert\pi\vert$, where $|\pi|$ the number of
blocks of $\pi$. However, this is far from optimal. The main
reason for a reduction of the exponent comes from the fact
that some of the indices which appear are actually used up
for matrix multiplication and thus do not contribute a
factor of $N$. For example, for
$\sigma=\{(2,3),(4,5),\cdots,(2m,1)\}$ one has
\begin{align*}
S_\sigma(N)&=\sum_{\substack{j_1,\dots,j_{2m}=1\\
j_2=j_3,j_4=j_5,\cdots,j_{2m}=j_1}}^N t^{(1)}_{j_1j_2}t^{(2)}_{j_3j_4}\cdots
t^{(m)}_{j_{2m-1}j_{2m}}\\& = \sum_{i_1,\dots,i_m=1}^N
t^{(1)}_{i_1i_2}t^{(2)}_{i_2i_3}\cdots t^{(m)}_{i_{m}i_{1}}\\&=\Tr(T_1\cdots T_m),
\end{align*}
thus
\[
\vert S_\sigma(N)\vert \leq N \Vert T_1\cdots T_m\Vert\leq N
\prod_{k=1}^m\Vert T_k\Vert.
\]
Hence the trivial estimate $\rexp(\sigma)=m$ can here
actually be improved to $\rexp(\sigma)=1$.

Other cases, however, might not be so clear. For example,
what would one expect for
\begin{multline}\label{eq:tau}
\tau=\{(1),(2,4,11),(3,5,10),(6,7,8)\\(9,12,14,16,20),(13,15,17,18),(19,22,24),(21,23)\}.
\end{multline}
The corresponding sum $S_\tau$ is
\[
\sum_{j_1,\dots,j_{24}=1}^N
t^{(1)}_{j_1j_2}          t^{(2)}_{j_3j_4} 
t^{(3)}_{j_5j_6}          t^{(4)}_{j_7j_8}
t^{(5)}_{j_9j_{10}}       t^{(6)}_{j_{11}j_{12}} t^{(7)}_{j_{13}j_{14}}    t^{(8)}_{j_{15}j_{16}}    t^{(9)}_{j_{17}j_{18}}    t^{(10)}_{j_{19}j_{20}}   t^{(11)}_{j_{21}j_{22}}   t^{(12)}_{j_{23}j_{24}}
\]
subject to the constraints
\[\left\{
\begin{array}{l}
j_2 = j_4 = j_{11},\\   
j_3 = j_5 = j_{10}, \\
j_6 = j_7 = j_8, \\
j_9 = j_{12} = j_{14} = j_{16} = j_{20}, \\
j_{13} = j_{15} = j_{17} = j_{18}, \\
j_{19} = j_{22} = j_{24}, \\
 j_{21} = j_{23}
\end{array}\right.
\]
or, in terms of unrestricted summation indices:
\begin{equation} \label{eq:S-tau}
S_\tau=\sum_{i_1,i_2,\dots,i_8=1}^N
t^{(1)}_{i_1i_2}t^{(2)}_{i_3i_2}t^{(3)}_{i_3i_4}t^{(4)}_{i_4i_4}
t^{(5)}_{i_5i_3}t^{(6)}_{i_2i_5}t^{(7)}_{i_6i_5}t^{(8)}_{i_6i_5}t^{(9)}_{i_6i_6}
t^{(10)}_{i_7i_5}t^{(11)}_{i_8i_7}t^{(12)}_{i_8i_7}.
\end{equation}

The trivial estimate here is of order $N^8$, but it might
not be obvious at all that in fact the optimal choice is
$\rexp(\tau)=3/2$. The non-integer value in this case shows
that the problem does not just come down to a counting
problem of relevant indices.

We will show that there is an easy and beautiful algorithm
for determining the optimal exponent $\rexp(\pi)$ for any
$\pi$. Actually, it turns out that $\rexp(\pi)$ is most
easily determined in terms of a graph $G_\pi$ which is
associated to $\pi$ as follows. We start from the directed
graph $G_{0_{2m}}$ with $2m$ vertices $1,2,\dots,2m$ and
directed edges $(2,1),(4,3),\dots$, $(2m,2m-1)$. (This is
the graph which corresponds to unrestricted summation, i.e.,
to $\pi=0_{2m}$, where $0_{2m}$ is the minimal partition in
$\cP(2m)$ with $2m$ blocks, each consisting of one
element. The reason that we orient our edges in the
apparently wrong direction will be addressed in Remark
\ref{rem:orientation}.) Given a $\pi\in\cP(2m)$ we obtain
the directed graph $G_\pi$ by identifying in $G_{0_{2m}}$
the vertices which belong to the same blocks of $\pi$. We
will not identify the edges (actually, the direction of two
edges between identified vertices might be incompatible) so
that $G_\pi$ will in general have multiple edges, as well as
loops.

For example, the graph $G_\tau$ for $\tau$ from Equation
$\eqref{eq:tau}$ is given in Figure \ref{fig:G-tau}. It
should be clear how one can read off the graph $G_\tau$
directly from Equation \eqref{eq:S-tau}.

\begin{figure}\centering{
\includegraphics[width=9cm]{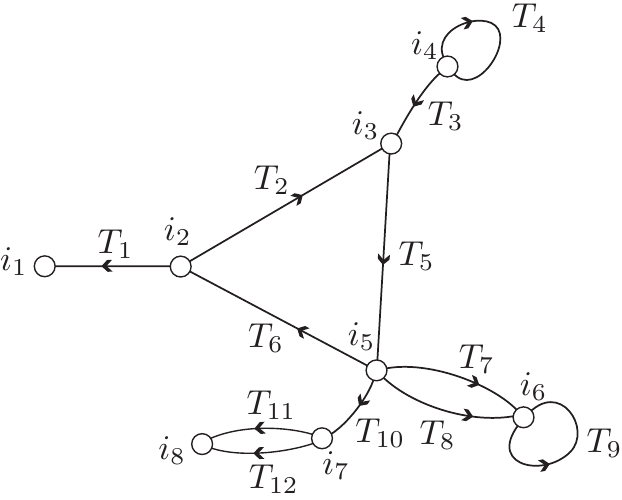}}
\caption{\label{fig:G-tau} The graph $G_\tau$ for the sum \eqref{eq:S-tau}.}
\end{figure}

The optimal exponent $\rexp(\pi)$ is then determined by the
structure of the graph $G_\pi$. Before we explain how this
works, let us rewrite the sum \eqref{eq:S-pi} more
intrinsically in terms of the graph $G=G_\pi$ as
\begin{equation}\label{eq:graph-sum-for-pi}
S_G(N):=\sum_{i : V \rightarrow [N]}\ \prod_{e\in E}\
t^{(e)}_{i_{t(e)},i_{s(e)}}.
\end{equation}
We sum over all functions $i: V \rightarrow [N]$ where $N =
\{ 1, 2, 3, \dots, N\}$, $V$ is the set of vertices of $G$, $E$ the set of edges,  and $s(e)$ and $t(e)$ denote the
source vertex and the target vertex of $e$
respectively. Note that we keep all edges through the
identification according to $\pi$, thus the $m$ matrices
$T_1,\dots,T_m$ in \eqref{eq:S-pi} show up in
\eqref{eq:graph-sum-for-pi} as the various $T_e$ for the $m$
edges of $G_\pi$.

\begin{remark}\label{rem:orientation}
Note that a factor of $t^{(l)}_{i_r i_s}$ in the sum in
(\ref{eq:graph-sum-for-pi}) produces an edge labelled $T_l$
starting at a vertex labelled $i_s$ and ending at a vertex
labelled $i_r$.
\begin{center}\includegraphics{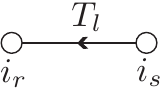}\end{center}
This reversing of the indices is an artifact of the usual
convention of writing $TS$ for the operator where one first
applies $S$ and then $T$.
\end{remark}

Clearly $\pi$ and $G_\pi$ contain the same information about
our problem; however, since the bound on $S_G(N)$ is easily
expressed in terms of $G_\pi$, we will in the following
forget about $\pi$ and consider the problem of bounding the
graph sum $S_G(N)$ in terms of $N$ for an arbitrary graph
$G$ with attached matrices. We will call a picture as in
Figure \ref{fig:G-tau} a \emph{graph of matrices}; for a
precise definition, see Definition
\ref{def:graph-of-matrices}.

\begin{example}
In the figures below we give four directed graphs and below
each graph the corresponding graph sum. One can see that if
the graph is a circuit then the graph sum is a trace of the
product of the matrices. However for more general graphs the
graph sum cannot easily be written in terms of traces.
Nevertheless, as Theorem \ref{thm:graph-sum} will show, there
is a simple way to understand the dependence of the graph
sum on $N$, the size of the matrices.

\begin{center}
$\vcenter{\hsize150pt%
\begin{center}
\includegraphics{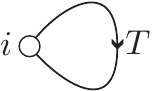}\\

\medskip

$\ds\Tr(T) = \sum_i t_{ii}$
\end{center}}$
$\vcenter{\hsize150pt%
\begin{center}
\includegraphics{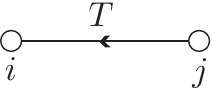}\\

\medskip

$\ds\sum_{i,j} t_{ij}$
\end{center}}$

\bigskip

$\vcenter{\hsize150pt%
\begin{center}
\includegraphics{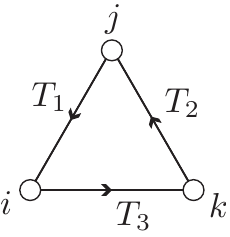}\\

\medskip

$\ds\Tr(T_1T_2T_3) = \sum_{i,j,k} t^{(1)}_{ij} t^{(2)}_{jk}
t^{(3)}_{ki}$
\end{center}}$
$\vcenter{\hsize150pt%
\begin{center}
\includegraphics{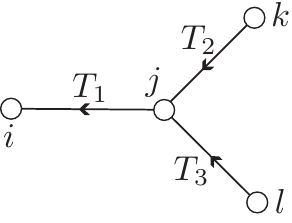}\\

\medskip

$\ds\sum_{i,j,k,l} t^{(1)}_{ij} t^{(2)}_{jk} t^{(3)}_{jl}$
\end{center}}$
\end{center}

\end{example}

The relevant feature of the graph is the structure of its
two-edge connected components.

\begin{notation}
1) A \emph{cutting edge} of a connected graph is an edge
whose removal would result in two disconnected subgraphs. A
connected graph is \emph{two-edge connected} if it does not
contain a cutting edge, i.e., if it cannot be cut into
disjoint subgraphs by removing one edge. A \emph{two-edge
  connected component} of a graph is a subgraph which is
two-edge connected and cannot be enlarged to a bigger
two-edge connected subgraph.

2) A \emph{forest} is a graph without cycles. A \emph{tree}
is a component of a forest, i.e., a connected graph without
cycles. A tree is \emph{trivial} if it consists of only one
vertex. A \emph{leaf} of a non-trivial tree is a vertex
which meets only one edge. The sole vertex of a trivial tree
will also be called a \emph{trivial leaf}.

\end{notation}

It is clear that if one takes the quotient of a graph with
respect to the two-edge connectedness relation (i.e., one
shrinks each two-edge connected component of a graph to a
vertex and just keeps the cutting edges), then one does not
have cycles any more, thus the quotient is a forest.

\begin{figure}\centering{
\includegraphics[width=5cm]{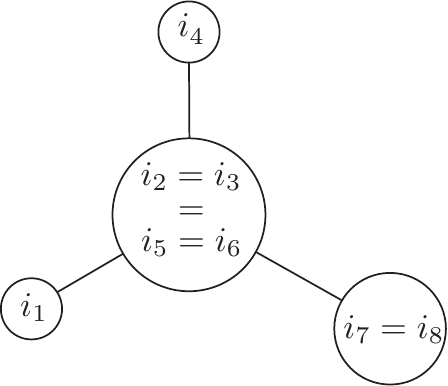}}
\caption{\label{fig:tau-forrest} The quotient graph
  $\cF(G_\tau)$ of Figure \ref{fig:G-tau}; the forest here
  consists of just one tree.}
\end{figure}

\begin{notation}
For a graph $G$ we denote by $\cF(G)$ its \emph{forest of
  two-edge connected components}: the vertices of $\cF(G)$
consist of the two-edge connected components of $G$ and two
distinct vertices of $\cF(G)$ are connected by an edge if
there is a cutting edge between vertices from the two
corresponding components in $G$.
\end{notation}

For the graph from Figure \ref{fig:G-tau} the corresponding
forest $\cF(G_\tau)$ is drawn in Figure
\ref{fig:tau-forrest}.

Now we can present our main theorem on bounds for sums of
the form \eqref{eq:graph-sum-for-pi}. In the special case of
a two-edge connected graph we obtain the same bound as
appears in the book of Bai and Silverstein \cite{BS}. In the
general case, however, our bound is less than that of \cite{BS}.

\begin{theorem}\label{thm:graph-sum}
$1)$ Let $G$ be a directed graph, possibly with multiple edges
and loops. Let for each edge $e$ of $G$ be given an $N\times
N$ matrix $T_e=(t^{(e)}_{ij})_{i,j=1}^N$.  Let $E$ and $V$,
respectively, be the edges and vertices of $G$ and
\begin{equation}\label{eq:graph-sum-for-pi-bis}
S_G(N) :=
\sum_{i:V \rightarrow [N]}\  \prod_{e \in E}\ 
t^{(e)}_{i_{t(e)},i_{s(e)}}
\end{equation}
where the sum runs over all functions $i : V \rightarrow
[N]$.  Then
\begin{equation}\label{eq:norm-sum-estimate}
\vert S_G(N)\vert \leq N^{\rexp(G)}\cdot \prod_{e\in E}\Vert T_e\Vert,
\end{equation}
where $\rexp(G)$ is determined as follows from the structure
of the forest $\cF(G)$ of two-edge connected components of
$G$:
$$\rexp(G)=\sum_{\text{\rm $\cl$ leaf of $\cF(G)$}} \rexp(\cl)$$
where
$$\rexp(\cl):=\begin{cases}
1,&\text{if $\cl$ is a trivial leaf}\\
\frac 12,&\text{if $\cl$ is a leaf of a non-trivial tree}
\end{cases}$$

$2)$ The bound in Equation \ref{eq:norm-sum-estimate} is
optimal in the following sense. For each graph $G$ and each
$N\in\NN$ there exist $N \times N$ matrices $T_e$ with
$\Vert T_e\Vert=1$ for all $e\in E$ such that
$$ S_G(N)=N^{\rexp(G)}.$$
\end{theorem}

\begin{example}\label{example}
Consider again our example $S_\tau$ from
\eqref{eq:S-tau}. Its forest $\cF(G_\tau)$, given in Figure
\ref{fig:tau-forrest}, consists of one tree with three
leaves; thus Theorem \ref{thm:graph-sum} predicts an order
of $N^{3/2}$ for the sum \eqref{eq:S-tau}.  In order to see
that this can actually show up (and thus give the main idea
for the proof of optimality), put all the matrices in Figure
\ref{fig:G-tau} for the non-cutting edges equal to the
identity matrix; then the problem collapses to the
corresponding problem on the tree, where we are just left
with the four indices $i_1,i_2,i_4,i_7$ and the three
matrices $T_1$, $T_3$, $T_{10}$. See Figure \ref{fig:G-in-out-1}.

\begin{figure}\centering{
\includegraphics[width=3cm]{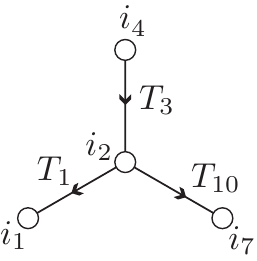}}
\caption{\label{fig:G-in-out-1} Putting the non-cutting edge
  matrices equal to the identity matrix reduces the problem
  for $G_\tau$ of Figure \ref{fig:G-tau} to this one.}
\end{figure}

The corresponding sum is
\[
S = \sum_{i_1, i_2, i_4, i_7 = 1}^N t^{(1)}_{i_1i_2}
t^{(3)}_{i_2i_4} t^{(10)}_{i_7i_2}
\]
Let $V$ now be the matrix
\begin{equation}\label{eq:V}
V=\frac 1{\sqrt N}\begin{pmatrix}
1 & 0 & \cdots & 0\\
1 & 0 & \cdots & 0\\
\vdots& \vdots&\ddots&\vdots\\
 1 & 0 & \cdots & 0
\end{pmatrix};
\end{equation}
and put $T_3=V^t$, $T_1=T_{10}=V$. Then $\Vert
T_1\Vert=\Vert T_3\Vert=\Vert T_{10}\Vert =1$ and we have
for this case

\[
S=\frac 1{N^{3/2}} \sum_{i_1,i_2,i_4,i_7=1}^N
\delta_{i_21}\delta_{i_21}\delta_{i_21}=\frac 1{N^{3/2}}
N^3=N^{3/2}.
\]
\end{example}

Note that each tree of the forest $\cF(G)$ makes a
contribution of at least 1 in $\rexp(G)$, because a
non-trivial tree has at least two leaves. One can also make
the above description more uniform by having a factor $1/2$
for each leaf, but counting a trivial leaf as actually two
leaves. (The reason for this special role of trivial leaves
will become apparent in the proof of Theorem
\ref{thm:graph-sum} in the next section.) Note also that the
direction of the edges plays no role in the estimate
above. The direction of an edge is only important in order
to define the contribution of an edge to the graph sum. One
direction corresponds to the matrix $T_e$, the other
direction corresponds to the transpose $T^t_e$. Since the
norm of a matrix is the same as the norm of its transpose,
the estimate is the same for all graph sums which correspond
to the same undirected graph.

Finally, we want to give an idea of our strategy for the
proof of Theorem \ref{thm:graph-sum}.  One of the main steps
consists in modifying the given graph of matrices (by
reversing some orientations, and by splitting some vertices
into two) in such a way that the corresponding sum $S_G(N)$
is not changed and such that the modified graph has the
structure of an input-output graph. By the latter we mean
that we have a consistent orientation of the graph from some
input vertices to some output vertices, see Definition
\ref{def:input-output}.
\begin{figure}\centering{
\includegraphics[width=6cm]{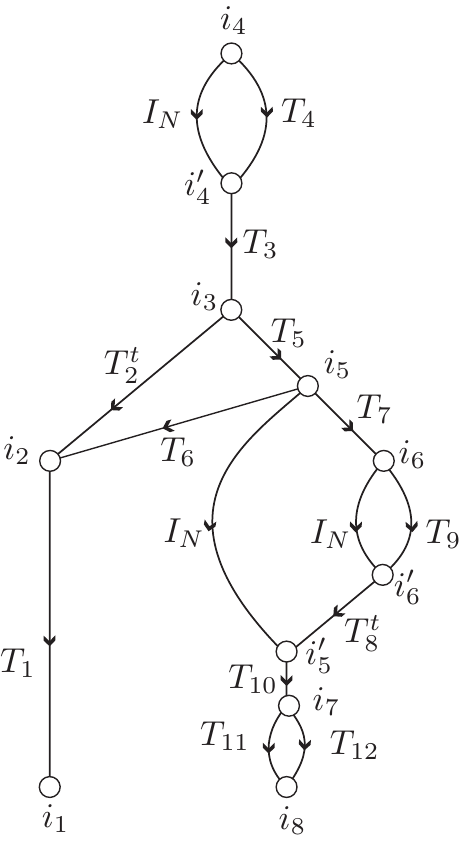}}
\caption{\label{fig:G-in-out} A modification of the graph
  $G_\tau$ from Figure \ref{fig:G-tau} in input-output
  form. Note that the input vertex $i_4$ and the output
  vertices $i_1$ and $i_8$ are chosen from the leaves of
  $\cF(G_\tau)$.}
\end{figure}

For example, a suitable modification of the graph $G_\tau$
is presented in Figure \ref{fig:G-in-out}. We have reversed
the orientation of two edges (but compensated this by taking
the adjoint of the attached matrices) and also split each of
the vertices $i_4$, $i_5$, $i_6$ into two copies. To take
care of the fact that in the summation we must have $i_4 =
i_4'$ we have added an additional edge between $i_4$ and
$i_4'$ with the identity matrix attached and similarly for
$i_5$ and $i_5'$ and $i_6$ and $i_6'$. So in order to obtain
a bound for $S_\tau$ it suffices to obtain a bound for the
graph $G$ from Figure \ref{fig:G-in-out}. But this has now a
kind of linear structure, with $i_4$ as input vertex and
$i_1$ and $i_8$ as output vertices. This structure allows us
to associate to the graph $G$ an operator $T_G$, which is
described in terms of tensor products of the maps $T_e$ and
partial isometries describing the splittings at the internal
vertices. $T_G$ maps from the vector space associated to
$i_4$ to the tensor product of the vector spaces associated
to $i_1$ and $i_8$. It is then fairly easy to see that the
norm of $T_G$ is dominated by the product of the norms of
the involved operators $T_e$; and the estimate for the sum
$S_G(N)$ is finally just an application of the
Cauchy-Schwarz inequality, where each of the input and
output vertices gives a factor $N^{1/2}$.

The rest of the paper is organized as follows. In Section
\ref{sect:two}, we formulate a slight generalization of our
theorem to rectangular matrices and introduce abstractly the
notion of a graph of matrices. Section \ref{sect:three}
deals with input-output graphs and the norm estimates for
their associated operators. In Section \ref{sect:four}, we
address the general case by showing how one can modify a
general graph of matrices to become an input-output
graph. Finally, in Section \ref{sect:five}, we generalize
the considerations from Example \ref{example} to show the
optimality of our choice for $\rexp(G)$.

%%%%%%%%%%%%%%%%%%%%%%%%%%%%%%%%%%%%%%%%%%%%%%%%%
%%%%%%%%%%%%%%%%%%%%%%%%%%%%%%%%%%%%%%%%%%%%%%%%%
%%                                             %%
%%                                             %%
%%   R E C T A N G U L A R   M A T R I C E S   %%
%%                                             %%
%%                                             %%
%%%%%%%%%%%%%%%%%%%%%%%%%%%%%%%%%%%%%%%%%%%%%%%%%
%%%%%%%%%%%%%%%%%%%%%%%%%%%%%%%%%%%%%%%%%%%%%%%%%

\section{Generalization to Rectangular Matrices}\label{sect:two}
Let us first formalize the input information for Theorem
\ref{thm:graph-sum}. We will deal here with the more general
situation of rectangular instead of square matrices. In
order for the graph sum to make sense we require that for a
given vertex $v$ all the matrices associated with an
incoming edge have the same number of rows, $N_v$ and
likewise all the matrices associated with an outgoing edge
have the same number of columns $N_v$. Moreover we shall
find it advantageous to treat the matrices as linear
operators between finite dimensional Hilbert spaces. So for
each vertex $v$ let $\HH_v = \C^{N_v}$ have the standard
inner product and let $\{\xi_1, \dots , \xi_{N_v} \}$ be the
standard orthonormal basis of $\C^{N_v}$.  Note that we use
the convention that inner products $(x,y)\mapsto \langle
x,y\rangle$ are linear in the second variable and we shall
use Dirac's bra-ket notation for rank one operators;
$|\xi\rangle\langle\eta|\, (\mu) = \langle \eta, \mu \rangle
\xi$.

\begin{definition}\label{def:graph-of-matrices}
A \emph{graph of matrices} consists of a triple
$\cG=(G,(\HH_v)_{v\in V},\ab(T_e)_{e\in E})$ in which 
\begin{enumerate}
\item $G = (V, E)$ is a directed graph (possibly with
  multiple edges and loops),
\item 
 $\HH_v$  is a finite dimensional Hilbert space  equal to $ \C^{N_v}$ and

\item    $T_e:\HH_{s(e)}\to\HH_{t(e)}$ is a linear operator.
\end{enumerate} (This is also known as a representation of a quiver, but we shall not need this terminology.)
\end{definition}

Here is the generalization of Theorem \ref{thm:graph-sum} to
the case of a rectangular matrices.

\begin{theorem}\label{thm:graph-sum-general}
Let  $\cG=(G,(\HH_v)_{v \in V},(T_e)_{e \in E})$ be a graph of matrices. Let
\begin{equation}\label{eq:graph-sum-general}
S(\cG):=\sum_{i: V \rightarrow \N} \ \prod_{e\in E}
\langle \xi_{i_{t(e)}},T_e\xi_{i_{s(e)}}\rangle.
\end{equation}
where the sum runs over all functions $i: V \rightarrow \N$
such that for each $v \in V$ we have $1 \leq i(v) \leq N_v$.

Let $\cF=\cF(G)$ be the forest of two-edge connected
components of $G$. Then
\begin{equation}\label{eq:norm-sum-estimate-general}
\vert S(\cG)\vert\leq \prod_{\text{\rm leaf $\cl$ of $\cF$}}
\left(\max_{v\in\cl}\dim\HH_v\right)^{\rexp(\cl)}\cdot \prod_{e\in E}\Vert T_e\Vert,
\end{equation}
where, for a leaf $\cl$, $v$ runs over all vertices in the
two edge connected component of $G$ corresponding to $\cl$,
and where
$$\rexp(\cl):=\begin{cases}
1,&\text{if $\cl$ is a trivial leaf,}\\
\frac 12,&\text{if $\cl$ is a leaf of a non-trivial tree.}
\end{cases}
$$
\end{theorem}

%%%%%%%%%%%%%%%%%%%%%%%%%%%%%%%%%%%%%%%%%%%%%%%%%%%%%%%%%%%
%%                                                       %%
%%   E S T I M A T E S  F O R  I N P U T - O U T P U T   %%
%%                                                       %%
%%                     G R A P H S                       %%
%%                                                       %%
%%%%%%%%%%%%%%%%%%%%%%%%%%%%%%%%%%%%%%%%%%%%%%%%%%%%%%%%%%%

\section{Estimate for Input-Output Graphs}\label{sect:three}

The main idea for proving the estimate
\eqref{eq:norm-sum-estimate-general} for a graph of matrices
is to first suppose that there is a flow from some vertices
designated \emph{input vertices}, $V_\inp$, to some other vertices
designated \emph{output vertices}, $V_\outp$, and then to show that
every graph can be modified to have such a flow.  All the
remaining vertices, which are neither input nor output
vertices, will be called \emph{internal vertices}. 

\begin{definition}\label{def:input-output}
Let $G$ be a directed graph (possibly with multiple
edges). We say that $G$ is an \emph{input-output graph} if
there exists two disjoint non-empty subsets, $V_{\inp}$ and
$V_{\outp}$, of the set of vertices of $G$ such that the
following properties are satisfied. \begin{itemize}
\item
$G$ does not contain a directed cycle. (Recall that a cycle
  is a closed path and that a path is directed if all the
  edges are oriented in the same direction.)

\item
Each vertex of $G$ lies on a directed path from some vertex
in $V_{\inp}$ to some vertex in $V_{\outp}$.
\item
Every internal vertex has at least one incoming edge and at
least one outgoing edge.
\item
Every input vertex has only outgoing edges and every output
vertex has only ingoing edges.
\end{itemize}
\end{definition}

Recall that $\{\xi_i \}_{i=1}^{N_v}$ is an orthonormal basis
for $\HH_v$. Let $V_0 \subset V$ be a subset, suppose that
we have a function $i: V_0 \rightarrow \N$ such that $i(v)
\leq N_v$ then for each $v \in V_0$, $\xi_{i_v}$ is an
element of our orthonormal basis of $\HH_v$. Thus an element
of our orthonormal basis of $\bigotimes_{v \in V_0} \HH_v$
is specified by a function $i : V_0 \rightarrow \N$ such
that $i(v) \leq N_v$ for each $v$ in $V_0$. When it is clear
from the context we shall just say that a basis element of
$\bigotimes_{v \in V_0} \HH_v$ is specified by a function
$i: V_0 \rightarrow \N$, but it should always be understood
that $i(v) \leq N_v$.

Hence if we form $\{ \otimes_{v \in V_0} \xi_{i_v} \}_i$
where $i$ runs over all functions $i: V_0 \rightarrow \N$,
we obtain an orthonormal basis of $\bigotimes_{v \in V_0}
\HH_v$. Thus an operator $T_{\cG}: \bigotimes_{v \in
  V_{\inp}} \HH_v \rightarrow \bigotimes_{w \in V_{\outp}}
\HH_w$ can be specified by giving
\[
\big\langle \myotimes_{w\in V_{\outp}}
\xi_{j_w},T_{\cG}\big(\myotimes_{v\in V_{\inp}}
\xi_{i_v}\big) \big\rangle
\]
for each basis vector $i: V_{\inp} \rightarrow \N$ and $j:
V_{\outp} \rightarrow \N$. In the theorem below we shall
show that a certain kind of graph sum can be written in
terms of a vector state applied to an operator defined by
the inner product above.  This is the first of two key steps
in proving Theorem~\ref{thm:graph-sum-general}.

\begin{theorem}\label{thm:graph-norm-estimate}
Let $\cG=(G,(\HH_v)_{v\in V},(T_e)_{e\in E})$ be a graph of
matrices and assume that $G$ is an input-output
graph with input vertices $V_\inp$ and output vertices
$V_\outp$.

\medskip\noindent 
$1)$ 
We define $T_{\cG}: \bigotimes_{v\in V_{\inp}}\HH_v \to
\bigotimes_{w \in V_{\outp}} \HH_w$ by
\begin{equation}\label{eq:graph-operator}
\langle 
\myotimes_{w\in V_{\outp}}
\xi_{j_w},T_{\cG}\big(\myotimes_{v\in V_{\inp}}
\xi_{i_v}\big) \rangle :=\sum_{k:V \rightarrow
  \N} \prod_{e\in E} \langle
\xi_{k_{t(e)}},T_e\xi_{k_{s(e)}}\rangle,
\end{equation}
where $i: V_{\inp} \rightarrow \N$, $j:
V_{\outp} \rightarrow \N$ and $k$ runs over all maps $k: V \rightarrow
\N$ such that $k|_{V_{\inp}} = i$ and $k|_{V_{\outp}} = j$.

Then we have
\begin{equation}\label{eq:graph-norm-estimate}
\Vert T_\cG\Vert\leq \prod_{e\in E}\Vert T_e\Vert.
\end{equation}

\noindent
$2)$ 
For the graph sum \eqref{eq:graph-sum-general} we have
\[
S(\cG)=\langle  \myotimes_{w\in V_{\outp}}
\xi^w,T_\cG\myotimes_{v\in V_{\inp}} \xi^v\rangle.
\]
where $\xi^v = \xi_1 + \cdots + \xi_{N_v} \in \HH_v$, and we have the
estimate
\begin{equation}\label{eq:graph-sum-input-output}
\vert S(\cG)\vert\leq \prod_{v\in V_\inp\cup V_\outp}
\dim(\HH_v)^{1/2}\cdot \prod_{e\in E}\Vert T_e\Vert.
\end{equation}
\end{theorem}

\begin{proof}

The key point is to observe that we can write the operator
$T_\cG$ as a composition of tensor products of the edge
operators $T_e$ and isometries corresponding to the internal
vertices. Every internal vertex has, by the definition of an
input-output graph, some incoming edges and some outgoing
edges, let's say $t$ incoming and $s$ outgoing (with
$t,s\geq 1$). Then the summation over the orthonormal basis
of $\HH_v$ for this internal vertex corresponds to an
application of the mapping $L_v: \HH^{\otimes t}_v \rightarrow
\HH^{\otimes s}_v$ given by
\[
L_v=\sum_{i=1}^{N_v}  \vert \xi_i^{\otimes s}\rangle\langle
\xi_i^{\otimes t}\vert.
\]
In terms of our basis we have for all $1 \leq i_1, \dots , i_t \leq N_v$
\[
L_v( \xi_{i_1} \otimes \cdots \otimes \xi_{i_t} ) =
\begin{cases}
\bigl(\xi_{i_1}\bigr)^{\otimes s} & 
\text{if $i_1 = \cdots =i_t$}\\
0 &  \text{otherwise.}
\end{cases}
\]
The mapping $L_v$ is, for all internal vertices $v$, a partial
isometry, and thus has norm equal to $1$.

It remains to put all the edge operators and the vertex
isometries together in a consistent way. For this, we have
to make sure that we can order the application of all these
operators in a linear way so that their composition
corresponds to the operator defined by
\eqref{eq:graph-operator}. However, this is guaranteed by
the input-output structure of our graph. We can think of our
graph as an algorithm, where we are feeding input vectors
into the input vertices and then operate them through the
graph, each edge doing some calculation, and each vertex
acting like a logic gate, doing some compatibility
checks. The main problem is the timing of the various
operations, in particular, how long one has to wait at a
vertex, before applying an operator on an outgoing edge. In
algorithmic terms, it is clear that one has to wait until
all the input information is processed; i.e. one has to wait
for information to arrive along the longest path from an
input vertex to the given vertex.

To formalize this, let us define a distance function
$d:V \to \{0,1,2,\dots\}$ on our graph $G$ which measures
the maximal distance from a vertex to a input vertex,
\[
d(v) := \max\left\{k\ \left|\ \vcenter{\hsize
  16em\raggedright there exists a directed path of length
  $k$ from some input vertex to $v$}\right.\right\}.
\]
The length of a path is the number of edges it uses. Note
that since an input vertex has no incoming edges, we have
$d(v)=0$ for all input vertices. The number $d(v)$ tells us
how long we should wait before we apply the isometry
corresponding to $v$; after $d(v)$ steps all information
from the input vertices has arrived at $v$.  Let $r$ be the
maximal distance (which is achieved for one of the output
vertices). The distance function $d$ gives us a
decomposition of the vertices $V$ of our graph into
disjoint level sets
\[
V_k:=\{v\in V \mid d(v)=k\},\qquad V=\bigcup_{k=0}^r
V_k.
\]
Note that, for any edge $e$, we have $d(t(e))\geq
d(s(e))+1$. In order to have a clearer notation it is
preferable if our edges connect only vertices which differ
in $d$ exactly by 1. This can easily be achieved by adding
vertices on edges for which this difference is bigger than
1. The new vertices have one incoming edge and one outgoing
edge. We have of course also to attach matrices to those
edges, and we do this in such a way that all incoming edges
of the new vertices get the identity matrix, the original
matrix $T_e$ is reserved for the last piece of our
decomposition. These new vertices will not change the
operator $T_{\cG}$ nor the graph sum $S(\cG)$. In the same
way we can insert some new vertices for all incoming edges
of the output vertices and thus arrange that every output
vertex $v$ has maximal possible distance $d(v)=r$. (Note
that there cannot be a directed path from one output vertex
to another output vertex, because an output vertex has no
outgoing edges.)

Thus we can assume without loss of generality that we have
$d(t(e))= d(s(e))+1$ for all edges $e\in E$ and that
$d(v)=r$ for all $v\in V_\outp$. We have now also a
decomposition of $E$ into a disjoint union of level sets,
\[
E_k:=\{e\in E \mid d(t(e))=k\},\qquad E = \bigcup_{k=1}^r
E_k.
\]
Edges from $E_k$ are connecting vertices from $V_{k-1}$ to
vertices from $V_k$.

Note that our Hilbert spaces correspond on one side to the
vertices, but on the other side also to the edges as source
and target Hilbert spaces; to make the latter clearer, let
us also write
\[
T_e:\HHs_e\to\HHt_e,
\]
where of course $\HHs_e$ is the same as $\HH_{s(e)}$ and
$\HHt_e$ is the same as $\HH_{t(e)}$.

We can now write
\begin{equation}\label{eq:operator-product}
T_\cG=L_r\cdot T_{r}\cdot L_{r-1}\cdot T_{r-1}\cdots L_1
\cdot T_1\cdot L_0,
\end{equation}
where $L_k$ is the tensor product of all partial isometries
corresponding to the vertices on level $k$, and $T_k$ is the
tensor product of all edge operators corresponding to the
edges on level $k$. More precisely,
\[
T_k:\bigotimes_{e\in E_k} \HHs_e \to \bigotimes_{e\in E_k}
\HHt_e
\]
is defined as
$$T_k:=\bigotimes_{e\in E_k} T_e;$$
whereas
$$L_k:=\bigotimes_{v\in V_k} L_v,$$
with the vertex partial isometry
\[
L_v:\bigotimes_{\substack{e\in E \\ t(e)=v}}\HHt_e\to
\bigotimes_{\substack{f\in E \\ s(f)=v}}\HHs_f
\]
given by
\[
L_v=\sum_{i = 1}^{N_v} 
\vert \xi_i^{\otimes s}\rangle
\langle \xi_i^{\otimes t}\vert,
\] 
where $s$ and $t$ are the number of edges which have $v$ as
their source and target, respectively.

Since we do not have incoming edges for $v\in V_\inp$ nor
outgoing edges for $v\in V_\outp$, one has to interpret
$L_0$ and $L_r$ in the right way. Namely, for $v\in V_\inp$,
the operator $L_v$ acts on
\[
L_v:\HH_v\to \bigotimes_{\substack{e\in E \\ s(e)=v}}\HHs_e
\]
given by
\[
L_v=\sum_{i=1}^{N_v} \vert \xi_i^{\otimes s}
\rangle\langle \xi_i\vert; 
\]
and similarly for $v\in V_\outp$.  (Formally, one can
include this also in the general formalism by adding one
incoming half-edge to each input vertex and one outgoing
half-edge to each output vertex.) With this convention, the
product given in \eqref{eq:operator-product} is an operator
from $\bigotimes_{v\in V_{\inp}}\HH_v$ to $\bigotimes_{w\in
  V_{\outp}}\HH_w$. It is clear that
\eqref{eq:operator-product} gives the same operator as
\eqref{eq:graph-operator}.

Now the factorization \eqref{eq:operator-product} and the
fact that all $L_v$ and thus all $L_k$ are partial
isometries yield
\[
\Vert T_\cG\Vert \leq \prod_{k=0}^r \Vert L_k\Vert\cdot
\prod_{k=1}^r \Vert T_k\Vert= \prod_{k=1}^r \Vert
T_k\Vert=\prod_{e\in E}\Vert T_e\Vert.
\]
This is the norm estimate \eqref{eq:graph-norm-estimate}
claimed for the operator $T_\cG$.

In order to get the estimate for the graph sum $S(\cG)$ we
have to note the difference between $T_\cG$ and $S(\cG)$:
for $T_\cG$ we sum only over the internal vertices and thus
remain with a matrix, indexed by the input and output
vertices; for $S(\cG)$ we also have to sum over these input
and output vertices. If we denote by
\[
\xi^v:=\sum_{i = 1}^{N_v} \xi_i \in \HH_v
\]
the sum over the vectors from our orthonormal basis of
$\HH_v$, then we have
\[
S(\cG)=\langle  \bigotimes_{w\in V_{\outp}}
\xi^w,T_\cG\bigotimes_{v\in V_{\inp}} \xi^v\rangle.
\]
An application of the Cauchy-Schwartz inequality yields then
\[
\vert S(\cG)\vert\leq \Vert T_\cG \Vert\cdot \prod_{v\in
  V_{\inp}} \Vert \xi^v\Vert\cdot \prod_{w\in V_{\outp}}
\Vert \xi^w\Vert.
\]
Since the norm of $\xi^v$ is, by Pythagoras's theorem, given
by $(\dim\HH_v)^{1/2}$, we get the graph sum estimate
\eqref{eq:graph-sum-input-output}.
\end{proof}

%%%%%%%%%%%%%%%%%%%%%%%%%%%%%%%%%%%%%%%%%%%%%%%%%%%%
%%                                                %%
%%   P R O O F  T H E  G E N E R A L  C A S E     %%
%%                                                %%
%%                                                %%
%%                                                %%
%%%%%%%%%%%%%%%%%%%%%%%%%%%%%%%%%%%%%%%%%%%%%%%%%%%%%

\section{Proof of the General Case}\label{sect:four}

Let us now consider a graph of matrices as in Theorem
\ref{thm:graph-sum-general}. The problem is that the
underlying graph $G$ might not be an input-output
graph. However, we have some freedom in modifying $G$
without changing the associated graph sum. First of all, we
can choose the directions of the edges arbitrarily, because
reversing the direction corresponds to replacing $T_e$ by
its transpose $T_e^t$. Since the norm of $T_e$ is the same as
the norm of $T_e^t$ the estimate for the modified graph will
be the same as the one for the old graph. More serious is
that, in order to apply Theorem
\ref{thm:graph-norm-estimate} we should also remove
directed cycles in $G$. This cannot, in general, be achieved
by just reversing some directions.  (As can clearly be seen
in the case of a loop.) The key observation for taking care
of this is that we can split a vertex $v$ into $v$ and $v'$
and redistribute at will the incoming and outgoing edges
from $v$ between $v$ and $v'$. We put one new edge $f$
between $v$ and $v'$ with the corresponding operator $T_f$
being the identity matrix. The constraint from $T_f$ in the
graph sum will be that after the splitting, the basis vector
for the vertex $v$ has to agree with the basis vector for
the vertex $v'$, so summation over them yields the same
result as summation over the basis of $\HH_v$ before the
splitting. Thus this splitting does not change the given
graph sum. Since the norm of the identity matrix is 1, this
modification will also not affect the wanted norm estimate.

One should of course also make sure that the forest
structure of the two-edge connected components is not
changed by such modifications. For the case of reversing
arrows this is clear; in the case of splitting vertices the
only problem might be that the new edge between $v$ and $v'$
is a cutting edge. This can actually happen, but only in the
case where $v$ constitutes a two-edge connected component by
itself. In that case, we do the splitting as before but add
two new edges between $v$ and $v'$, both with the same
orientation and both with the identity operator.

This motivates the following definition of the modification
of a graph of matrices.

\begin{definition}
We say that $\hat\cG=(\hat G,(\hat \HH_v)_{v\in \hat V},(\hat T_e)_{e\in \hat E})$ is a \emph{modification}
of $\cG=(G,(\HH_v)_{v\in V},(T_e)_{e\in E})$, if the
former can be obtained from the latter by finitely many
applications of the following operations:
\begin{itemize}
\item
change the direction of the arrow of an edge $e$ and replace
$T_e:\HH_{s(e)}\to \HH_{t(e)}$ by its transpose
$T_e^t:\HH_{t(e)}\to\HH_{s(e)}$
\item
split a vertex $v$ into two vertices $v$ and $v'$,
redistribute in some way the incoming and outgoing edges for
$v$ together with their matrices to $v$ and $v'$ and add a
new edge between $v$ and $v'$ with arbitrary direction for
this edge and the identity matrix attached to it; should $v$
be a two-edge connected component, then we add two edges
between $v$ and $v'$, both with the same orientation, and
both having the identity matrix attached to them
\end{itemize}
\end{definition}

Our discussion from above can then be summarized in the
following proposition.

\begin{proposition}
Let $\hat \cG=(\hat G,(\hat \HH_w)_{w\in \hat V},(\hat
T_f)_{f\in \hat E})$ be a modification of
$\cG=(G,(\HH_v)_{v\in V},(T_e)_{e\in E})$. Then we
have:
\begin{itemize}
\item
the graph sums are the same,
$$S(\cG)=S(\hat\cG);$$
\item
the forests of two-edge connected components are the same,
$$\cF(G)=\cF(\hat G);$$
\item
the product of the norm of the edge operators is the same,
$$\prod_{e\in E} \Vert T_e\Vert=\prod_{f\in \hat E}\Vert T_f\Vert.$$
\end{itemize}
Thus, in order to show the graph sum estimate
\eqref{eq:norm-sum-estimate-general} for $\cG$ it is enough
to prove this estimate for some modification $\hat\cG$.
\end{proposition}

So the crucial step for the proof of Theorem
\ref{thm:graph-sum-general} is now to modify a given graph
$G$ to an input-output graph $\hat G$ with the right number
of input and output vertices.

\begin{proposition}
Let $\cG$ be a graph of matrices. Then there exists a
modification $\hat \cG$ such that the underlying graph $\hat
G$ of the modification is an input-output
graph.

Furthermore, the input and output vertices can be
chosen such that: for each non-trivial tree of the forest 
$\cF(G) (= \cF(\hat G))$ we have one leaf
as input leaf and all the other leaves as output leaves.
For a trivial tree, the trivial leaf is considered both as
input and output leaf. The input vertices of $\hat G$
shall consist of one vertex from each input leaf, and the
output vertices shall consist of one vertex from each output leaf.
\end{proposition}

\begin{proof}
Clearly we can assume that the underlying graph $G$ of $\cG$
is connected, because otherwise we do the following
algorithm separately for each connected component.

For such a connected $G$, consider the tree of its two-edge
connected components. Declare arbitrarily one leaf as
\emph{input leaf}, all the other leaves as \emph{output
  leaves}; if the tree is trivial, we declare its only leaf
both as input and output leaf. Furthermore, we choose an
arbitrary vertex from the input leaf as input vertex, and 
for our output vertices we choose an
arbitrary vertex from each output leaf. The direction from input leaf to output leaves
defines uniquely a flow in our tree from the input leaf to
the output leaves, i.e., this gives us a direction for the
cutting edges of $G$.

For each two-edge connected component we define now one
input vertex and one output vertex. For the input leaf we
have already chosen the input vertex; its output vertex is
the source vertex of one (arbitrarily chosen) of the
outgoing cutting edges. For the output leaves we have
already chosen their output vertices; as input vertex we
take the target vertex of the (unique) incoming cutting
edge. For all the other, non-leaf, components we choose the
target vertex of the (unique) incoming cutting edge as input
vertex and the source vertex of one (arbitrarily chosen) of
the outgoing cutting edges as the output vertex. We want all
those input and output vertices to be different, which can
be achieved by splitting, if necessary, some of them into
two.

So now each two-edge connected component has one input
vertex and one output vertex. If we are able to modify each
two edge connected component in such a way that it is an
input-output graph with respect to its input and output
vertex, then by putting the two-edge connected components
together and declaring all input vertices but the one from
the input leaf and all output vertices but the ones from the
output leaves as internal indices, we get the modification
$\hat G$ with the claimed properties. It only remains to do
the modification of the two-edge connected components. This
will be dealt with in the next lemma.
\end{proof}

\begin{lemma}
Let $\cG$ be a graph of matrices and assume that the
underlying graph $G$ is two-edge connected. Let $v$ and $w$
be two disjoint vertices from $G$. Then there exists a
modification $\hat\cG$ of $\cG$, such that the underlying
graph $\hat G$ of the modification is an input-output graph,
with input vertex $v$ and output vertex $w$.
\end{lemma}

\begin{proof}
The proof of this can be found in \cite[Ch. 11]{BS}. Let us
recall the main steps. One builds a sequence $G_k$ of
input-output graphs (all with $v$ as input vertex and $w$ as
output vertex) such that each step is manageable and that
the last graph is the wanted one. For this construction we
ignore the given orientation of the edges of $G$, but will
just use the information from $G$ as undirected graph; then
we will choose convenient orientations for the edges when
constructing the sequence $G_k$.

First, we choose a simple path (i.e., a path without
cycles), in our graph $G$ from $v$ to $w$. We direct all
edges on this path from $v$ to $w$. This path with this
orientation of edges is our first input-output graph $G_1$.

Assume now we have constructed an input-output graph
$G_k$. If this is not yet the whole graph, then we can
choose an edge $e$ which is not part of $G_k$ and which has
one of its vertices, say $x$, on $G_k$. Let us denote the
other vertex of $e$ by $z$. Then one can find a simple path
in $G$ which connects $z$ with $G_k$ and does not use
$e$. (This is possible, because otherwise $e$ would be a
cutting edge.) Denote the end point of this path (lying on
$G_k$) by $y$. (Note that $y$ might be the same as $z$.) We
have now to direct this path between $x$ and $y$. If
$x\not=y$, then there was:
\begin{enumerate}
\item
either a directed path from $x$ to $y$ in $G_k$, in which
case we direct the new path also from $x$ and $y$;
\item
or a directed path from $y$ to $x$ in $G_k$, in which case
we direct the new path also from $y$ and $x$;
\item
or there was no such path in $G_k$, in which case we can
choose any of the two orientations for the new path between
$x$ and $y$.
\end{enumerate}
(Note that the first and second case cannot occur
simultaneously, because otherwise we would have had a
directed cycle in $G_k$.)

The only problematic case is when $x=y$, i.e., when the new
path is actually a cycle. In this case we split the vertex
$x=y$ into two different vertices, $x$ and $y$; $x$ gets all
the incoming edges from $G_k$ and $y$ gets all the outgoing
edges from $G_k$, and the new edge is directed from $x$ to
$y$. Furthermore, the new cycle becomes now a directed path
from $x$ to $y$.

Our new graph $G_{k+1}$ is now given by $G_k$ (possibly
modified by the splitting of $x$ into $x$ and $y$) together
with the new path from $x$ to $y$. It is quite easy to see
that $G_{k+1}$ is again an input-output graph, with the same
input vertex and output vertex as $G_k$.

We repeat this adjoining of edges until we have exhausted
our original graph $G$, in which case our last input-output
graph is the wanted modification.
\end{proof}

%%%%%%%%%%%%%%%%%%%%%%%%%%%%%%%%%%%%%%%%%%%%%%%%%%%%%%%%%%%%
%%%%%%%%%%%%%%%%%%%%%%%%%%%%%%%%%%%%%%%%%%%%%%%%%%%%%%%%%%%%
%%                                                        %%
%%                                                        %%
%%         P R O O F   O F   O P T I M A L I T Y          %%
%%                                                        %%
%%                                                        %%
%%%%%%%%%%%%%%%%%%%%%%%%%%%%%%%%%%%%%%%%%%%%%%%%%%%%%%%%%%%%
%%%%%%%%%%%%%%%%%%%%%%%%%%%%%%%%%%%%%%%%%%%%%%%%%%%%%%%%%%%%

\section{Proof of Optimality}\label{sect:five}
In order to show the second part of Theorem
\ref{thm:graph-sum}, that our exponent $\rexp(G)$ is
optimal, we just have to adapt the corresponding
considerations in Example \ref{example} to the general
case. For a given graph we attach to each non-cutting edge
the identity matrix; thus all indices in a two-edge
connected component of $G$ get identified and we reduce the
problem to the case that $G$ is a forest. Since it suffices
to look on the components separately, we can thus assume
that $G$ is a tree. If this tree is trivial, then we have no
cutting edges left and we clearly get a factor $N$.
Otherwise, we put an orientation on our tree by declaring
one leaf as input leaf and all the other leaves as output
leaves. Then we attach the following matrices to the edges
of this tree
\[
T_e=\begin{cases} V^t,& \text{if $e$ joins the input leaf
  with an internal vertex}\\ V,&\text{if $e$ joins an output
  leaf with an internal vertex}\\ 1,&\text{otherwise}
\end{cases},
\]
where $V$ is the matrix given in \eqref{eq:V}. Again, it is
straightforward to see that this choice forces every index
corresponding to an internal vertex to be equal to 1,
whereas there is no restriction for the indices
corresponding to the leaves; taking into account also the
$1/\sqrt N$ factors from the operators $V$, we will get in
the end $N^{\#\text{leaves}/2}$ for the sum.

\section*{acknowledgments}
The authors would like to thank Van Vu for his comments on a preliminary version of the paper.

\end{document}